\documentclass{article}

\usepackage{amsmath}
\usepackage{amssymb,amsthm}
\numberwithin{equation}{section}
\newtheorem{thm}{Theorem}[section]
\newtheorem*{thm*}{Theorem}

\newtheorem{cor}[thm]{Corollary}
\newtheorem{prop}[thm]{Proposition}
\newtheorem{lem}[thm]{Lemma}
\theoremstyle{definition}
\newtheorem{defn}[thm]{Definition}
\theoremstyle{remark}
\newtheorem{rmk}[thm]{Remark}
\newtheorem{exam}[thm]{Example}

\usepackage[all]{xy}

\usepackage[T1]{fontenc}
\usepackage{kpfonts,baskervald}

\usepackage[usenames]{color}
\usepackage{todonotes}
\usepackage{tikz-cd}
\usepackage{hyperref}
\hypersetup{
  colorlinks   = true, 
  urlcolor     = blue, 
  linkcolor    = blue, 
  citecolor   = red 
}

\newcommand{\co}{\colon\thinspace}
\newcommand{\mb}[1]{\mathbb{#1}}
\newcommand{\mf}[1]{\mathfrak{#1}}

\newcommand{\Sp}{{\mathcal{S}p}}
\newcommand{\Spa}{{\mathcal{S}}}

\newcommand{\op}{{op}}
\newcommand{\too}{\xrightarrow}

\newcommand{\CC}{{\mathcal{C}}}
\newcommand{\DD}{{\mathcal{D}}}
\newcommand{\EE}{{\mathcal{E}}}
\newcommand{\FF}{{\mathcal{F}}}
\newcommand{\KK}{{\mathcal{K}}}
\newcommand{\MM}{{\mathcal{M}}}
\newcommand{\NN}{{\mathcal{N}}}
\newcommand{\OO}{{\mathcal{O}}}
\newcommand{\PP}{{\mathcal{P}}}

\newcommand{\TT}{{\mathcal{T}}}
\newcommand{\LM}{{\mathcal{LM}}}
\newcommand{\Ups}{\Upsilon}
\newcommand{\Op}{{\mathop{Op}}}
\newcommand{\wt}{\widetilde}

\DeclareMathOperator{\Sq}{Sq}

\DeclareMathOperator{\Map}{Map}

\DeclareMathOperator{\End}{End}
\DeclareMathOperator{\Pow}{Pow}

\DeclareMathOperator{\Mod}{Mod}

\DeclareMathOperator{\CAlg}{CAlg}
\DeclareMathOperator{\Alg}{Alg}

\DeclareMathOperator{\Fun}{Fun}

\DeclareMathOperator{\Sym}{Sym}
\DeclareMathOperator{\Mul}{Mul}
\DeclareMathOperator{\Free}{Free}
\DeclareMathOperator{\ev}{ev}
\DeclareMathOperator{\StNat}{StNat}

\DeclareMathOperator{\LMod}{LMod}

\DeclareMathOperator{\Tr}{Tr}
\DeclareMathOperator{\Res}{Res}

\DeclareMathOperator*{\colim}{colim}
\DeclareMathOperator*{\hocolim}{hocolim}
\DeclareMathOperator*{\holim}{holim}

\DeclareMathOperator*{\sma}{\wedge}

\title{Stable power operations}
\author{Saul Glasman, Tyler Lawson}
\begin{document}
\maketitle

\begin{abstract}
  For any $E_\infty$ ring spectrum $E$, we show that there is an
  algebra $\Pow(E)$ of stable power operations that acts naturally on
  the underlying spectrum of any $E$-algebra. Further, we show that
  there are maps of rings $E \to \Pow(E) \to \End(E)$, where the latter
  determines a restriction from power operations to stable operations
  in the cohomology of spaces. In the case where $E$ is the mod-$p$
  Eilenberg--Mac Lane spectrum, this realizes a natural quotient from
  Mandell's algebra of generalized Steenrod operations to the mod-$p$
  Steenrod algebra. More generally, this arises as part of a
  classification of endomorphisms of representable functors from an
  $\infty$-category $\CC$ to spectra, with particular attention to the
  case where $\CC$ is an $\OO$-monoidal $\infty$-category.
\end{abstract}

\section{Introduction}

The Steenrod reduced power operations occupy the unusual position that
they are often constructed in one way and used in another. We
typically construct the Steenrod operations as {\em power operations}:
the diagonal map $X \to X^p$ is equivariant for the action of the
symmetric group $\Sigma_p$, and this produces cochain-level structure
that results in cohomology-level operations. We typically think of the
Steenrod operations as {\em stable operations}: they generate all
possible stable natural transformations between cohomology groups, and
as such are represented by maps between Eilenberg--Mac Lane
spectra. In some sense, these are opposing perspectives. For example,
if $x$ is an element in $H^n(X;\mb F_2)$ then each of the identities
$Sq^n(x) = x^2$ and $Sq^0(x) = x$ is obvious from one of these two
points of view and not the other.

This divide becomes even more apparent if we replace ordinary mod-$p$
cohomology with a generalized cohomology theory $E^*$, represented by
a spectrum $E$. The stable operations on $E$-cohomology of degree $n$
are in bijection with homotopy classes of maps $E \to \Sigma^n E$, and
form a graded algebra. Power operations are more subtle and need more
from $E$. They require a refined multiplicative structure on $E$ to
define (e.g. see \cite{bmms-hinfty}) and have a more delicate
composition structure which is only distributive on one side
\cite{bergner-multisorted, borger-wieland-plethystic, rezk-wilkerson},
but they appear in more contexts: they also appear in the $E$-homology
of infinite loop spaces and the $E$-homology of commutative ring
spectra. Restricting our attention to those power operations that
preserve addition (called {\em additive} or {\em primitive}), we
obtain a graded algebra. In the case where $E$ is mod-$p$ cohomology,
this is called the algebra of (Araki--Kudo--)Dyer--Lashof operations
or simply the Dyer--Lashof algebra.

In terms of stable homotopy theory, we may identify $E^n(X)$ with the
homotopy group $\pi_{-n}$ of the function spectrum
$F(\Sigma^\infty_+ X, E) = E^X$. From this point of view, these groups
have stable operations because $F(\Sigma^\infty_+ X,E)$ is a left
module over the endomorphism algebra $F(E,E) = \End(E)$. On the other
hand, if $E$ is a commutative ring spectrum these groups have power
operations because they are the homotopy groups of the commutative
$E$-algebra $E^X$. These two roles played by $E^X$ are encapsulated in
a diagram of categories that commutes up to natural isomorphism:
\[
\xymatrix@C=4pc{
\Spa^\op \ar[r]^-{F(\Sigma^\infty_+(-),E)} \ar_{E^{(-)}}[d] &
\Mod_{\End(E)} \ar[d] \\
\CAlg_E \ar[r] & \Mod_E
}
\]
Here $\Spa$ is the category of spaces, $\Mod_R$ is the category of
left $R$-modules, and $\CAlg_E$ is the category of commutative
$E$-algebras.

In the case of ordinary mod-$p$ cohomology, we have the Steenrod
algebra $\mf{A}^*$ of stable operations on mod-$p$ cohomology and the
algebra $\mf{B}^*$ of generalized Steenrod operations \cite[\S
5]{mandell-einftypadic} acting on the homotopy of any commutative
algebra over the Eilenberg-Mac Lane spectrum $H\mb F_p$. For the
algebra $(H\mb F_p)^X$, these actions are related: there is a quotient
map of algebras $\mf{B}^* \to \mf{A}^*$ that annihilates the two-sided
ideal generated by a certain element $(1 - P^0)$.

Our goal is to systematically isolate what distinguishes the Steenrod
operations and Araki--Kudo--Dyer--Lashof operations as special among
power operations and gives them a relation to stable operations. For a
given commutative $E$, we will construct a ring spectrum $\Pow(E)$ of
\emph{stable power operations} that acts on the underlying spectrum of
any commutative $E$-algebra, together with a map $\Pow(E) \to
\End(E)$ that determines how stable power operations act on the
$E$-cohomology of spaces. Here are more precise statements of our
main results.
\begin{thm}\label{thm1.1}
  For any commutative ring spectrum $E$, there is an associative ring
  spectrum $\Pow(E)$ with a diagram of associative ring spectra
  \[
    E \to \Pow(E) \to \End(E).
  \]

  The ring $\Pow(E)$ has a natural action on the underlying
  spectrum of a commutative $E$-algebra in a matter compatible with
  stable cohomology operations, in the sense that there is a canonical
  lift in the following diagram:
  \[
    \xymatrix@C=4pc{
      \Spa^\op \ar[dd]_{E^{(-)}} \ar[r]^-{F(\Sigma^\infty_+(-),E)} &
      \Mod_{\End(E)} \ar[d] \\
      & \Mod_{\Pow(E)} \ar[d] \\
      \CAlg_E \ar[r] \ar@{.>}[ur] &
      \Mod_E
    }
  \]
\end{thm}
These three module categories, though initially appearing for quite
different reasons, are connected in the following way. The ring
$\End(E)$ is the ring of endomorphisms of the functor
$E^{(-)}\co \Spa^\op \to \Sp$; $E$ is the ring of endomorphisms of the
forgetful functor $\Mod_E \to \Sp$; and $\Pow(E)$ is the ring of
endomorphisms of the forgetful functor $\CAlg_E \to \Sp$. Moreover,
these three functors are representable: the representing objects are
cospectra.
\begin{itemize}
\item The functor $\Spa^\op \to \Sp$ is represented by the spectrum
  $E = \{E_n\}_{n \in \mb Z}$ itself. This is a spectrum object in
  $\Spa$ and hence a cospectrum object in $\Spa^\op$.
\item The forgetful functor $\Mod_E \to \Sp$ is represented by the
  desuspensions $\{\Sigma^n E\}_{n \in \mb Z}$. These form a cospectrum
  object in $\Mod_E$.
\item The forgetful functor $\CAlg_E \to \Sp$ is represented by the
  free $E$-algebras $\{\mb P_E(\Omega^n E)\}_{n \in \mb Z}$. These
  form a cospectrum object in $\CAlg_E$.
\end{itemize}
Knowledge of these representing objects enables us to calculate
endomorphism rings. In the case of mod-$p$ homology, we have the
following.

\begin{thm}
  For any prime $p$, the map $\Pow({H\mb F_p}) \to \End(H\mb
  F_p)$ of ring spectra becomes, on taking coefficient rings, a
  quotient map $(\mf{B}^*)^\wedge \to \mf{A}^*$ to the Steenrod
  algebra from the completion of the algebra of generalized Steenrod
  operations with respect to the excess filtration.
\end{thm}

As one benefit of having a spectrum level lift of the algebra
$\mf{B}^*$, we can employ several tools from the theory of secondary
operations on the homotopy groups of ring and module spectra
\cite{baues-muro-cupone, secondary}. This means that we can
immediately define operations within the Dyer--Lashof algebra, obtain
relations such as Peterson--Stein relations, and establish
compatibility of all of this structure with secondary Steenrod
operations.

We note that the construction of $\Pow(E)$ is outlined in course notes
of Lurie \cite[Lecture~24]{lurie-sullivanlectures}, where it is used
to show that power operations are compatible with geometric
realization of simplicial objects.

\subsection{Outline}

We will begin in \S\ref{sec:endomorphisms} and \S\ref{sec:presentable}
by recalling a number of results from Lurie's work on functoriality of
endomorphism objects, presentable categories, and the Yoneda
embedding. The goal is the analogue of a result from basic category
theory: if a category $\DD$ is enriched in $\mathcal{V}$, then the
enriched homomorphism objects $F(-,y)$ naturally take values in the
category of left modules over the algebra $\End(y) = F(y,y)$ in
$\mathcal{V}$. In \S\ref{sec:yoneda} we will further discuss how
spectrum objects in $\CC$ represent functors $\CC^\op \to \Sp$, and
that a left $A$-module structure on a spectrum object $Y$ is
essentially equivalent to a lift of its represented functor to left
$A$-modules. In \S\ref{sec:stable-end} we will discuss how to
calculate the universal example: the endomorphism algebra $\End(Y)$.

In \S\ref{sec:algebras-operads} we will apply this to the case of
stable presentable $\OO$-monoidal $\infty$-categories and algebras in
them. The spectra parametrizing stable operations turn out to be
related to certain generalized symmetric power functors. We will
briefly discuss the relation between stable operations and dual
Goodwillie calculus.

In \S\ref{sec:spaces} we will give our main results regarding
endomorphisms of commutative algebras and their relationship to
stable cohomology operations. In \S\ref{sec:weight},
\S\ref{sec:p-power}, and \S\ref{sec:homology} we will give example
calculations and show, in particular, the relationship with classical
Dyer--Lashof operations. For this we will need some equivariant
homotopy theory, which is carried out in the appendix
\S\ref{sec:equivariant-results}.

\subsection{Acknowledgements}

The authors would like to thank 
Tobias Barthel,
John Greenlees,
Nick Kuhn,
Akhil Mathew,
Haynes Miller,
Charles Rezk,
and
Dylan Wilson
for discussions related to this work.

The second author was partially supported by NSF grant 1610408 and a
grant from the Simons Foundation. He would like to thank the Isaac
Newton Institute for Mathematical Sciences for support and hospitality
during the programme HHH when work on this paper was undertaken. This
work was supported by: EPSRC grant numbers EP/K032208/1 and
EP/R014604/1.

\section{Functoriality of endomorphisms}
\label{sec:endomorphisms}

We will start with a brief rundown of results from
\cite{lurie-higheralgebra} on endomorphism objects.

\begin{itemize}
\item Operads and colored operads are encoded, in coherent category
  theory, by the notion of an $\infty$-operad \cite[\S
  2]{lurie-higheralgebra}.
\item To the associative operad (whose algebras are associative
  algebras $A$), and to the left-module colored operad (whose algebras
  are pairs $(A,M)$ of an algebra and a left $A$-module), there are
  associated $\infty$-operads $Assoc^\otimes$ and $\LM^\otimes$
  \cite[4.1.1.3, 4.2.1.7]{lurie-higheralgebra}.
\item A monoidal $\infty$-category $\CC$ is expressed in terms of a
  coCartesian fibration $\CC^\otimes \to Assoc^\otimes$, and an
  $\infty$-category $\MM$ left-tensored over $\CC$ is expressed as a
  coCartesian fibration $\MM^\otimes \to \LM^\otimes$
  \cite[4.1.1.10, 4.2.1.19]{lurie-higheralgebra}.
\item A lax monoidal functor $\CC \to \DD$ is expressible as a map
  $\CC^\otimes \to \DD^\otimes$ over $Assoc^\otimes$, and similarly a
  lax functor $\MM \to \NN$ of categories left-tensored over $\CC$ and
  $\DD$ is expressible as a map $\MM^\otimes \to \NN^\otimes$ over
  $\LM^\otimes$. Strong monoidal functors and functors strongly
  compatible with left-tensoring are expressed in terms of
  coCartesian morphisms \cite[2.1.3.6, 2.1.3.7]{lurie-higheralgebra}.
\item Right adjoints to strong monoidal functors and strongly
  compatible left-tensorings are lax
  \cite[7.3.2.7]{lurie-higheralgebra}.
\item Associative algebra objects $A$ in $\CC$ are defined as certain
  sections of the map $\CC^\otimes \to Assoc^\otimes$, and pairs of an
  associative algebra $A$ in $\CC$ and a left $A$-module $M$ in $\MM$
  are defined as certain sections of the maps $\MM^\otimes \to
  \LM^\otimes$ \cite[2.1.3.1, 4.1.1.6,
  4.2.1.13]{lurie-higheralgebra}. The category of algebras is
  denoted $\Alg(\CC)$ and the category of algebra-module pairs is
  denoted $\LMod(\MM)$, admitting forgetful functors to $\CC$ and $\CC
  \times \MM$ respectively.
\item In terms of these sections, a lax monoidal functor $\CC \to \DD$
  therefore induces a functor $\Alg(\CC) \to \Alg(\DD)$, and a lax
  tensored functor induces a compatible functor $\LMod(\MM) \to
  \LMod(\NN)$.
\item Further, a left-tensored category $\MM$ is enriched over $\CC$
  if, for any objects $M$ and $M'$ in $\MM$, there exists a function
  object $F_\MM(M,M')$ in $\CC$ such that maps $C \otimes M \to M'$
  are equivalent to maps $C \to F_\MM(M, M')$. Function objects are
  functorial in $M$ and $M'$ if they exist \cite[4.2.1.28,
  4.2.1.31]{lurie-higheralgebra}.
\item As a result, a lax monoidal functor $f\co \CC \to \DD$ with a
  compatible lax functor $g\co \MM \to \NN$ of enriched
  categories induces canonical maps $f (F_\MM(M,M')) \to F_\NN(gM, gM')$.
\item If function objects exist, then for any fixed $M$ the
  endomorphism object $\End(M) = F_\MM(M,M)$ has a canonical algebra
  structure, and the category $\LMod(\MM) \times_{\MM} \{M\}$ of
  algebras acting on $M$ is equivalent to the category of algebras
  $\Alg(\CC)_{/ \End(M)}$ of algebras with a map $A \to \End(M)$
  \cite[4.7.1.40, 4.7.1.41]{lurie-higheralgebra}.
\item As a result, a lax monoidal functor $f\co \CC \to \DD$ with a
  compatible lax functor $g\co \MM \to \NN$ of enriched categories
  induces canonical maps of algebras $f (\End(M)) \to \End(gM)$,
  compatible with the $\End(M)$-module structure on $M$ and the
  $\End(gM)$-module structure on $\End(gM)$.
\end{itemize}

\section{Presentable categories and Yoneda embeddings}
\label{sec:presentable}

We will now discuss some results on presentable $\infty$-categories
and their tensor product.

\begin{itemize}
\item There exist an $\infty$-category $\Pr^L$ (resp. $\Pr^R$), whose
  objects are presentable $\infty$-categories and colimit-presenting
  (resp. accessible and limit-preserving) functors. The assignment
  that takes a colimit-preserving functor to a limit-preserving
  right adjoint gives an antiequivalence between these categories:
  $\Fun^L(\CC, \DD) \cong \Fun^R(\DD,\CC)$ \cite[5.5.3.1,
  5.5.3.4]{lurie-htt}.
\item There is a symmetric monoidal structure on $\Pr^L$ such that
  functors $\CC \otimes \DD \to \EE$ are equivalent to functors
  $\CC \times \DD \to \EE$ that preserve colimits in each variable
  separately, and there is an equivalence
  $\Ups\co \CC \otimes \DD \cong \Fun^R(\CC^\op, \DD)$ \cite[4.8.1.15,
  4.8.1.17]{lurie-higheralgebra}.
\item A monoidal presentable category $\CC$ is equivalent to
  a monoid in $\Pr^L$, and a presentable category $\MM$ left-tensored over
  $\CC$ is equivalent to a left module over $\CC$ in $\Pr^L$. Under
  these conditions, $\MM$ is enriched over $\CC$
  \cite[4.2.1.33]{lurie-higheralgebra}.
\item For a monoidal presentable $\infty$-category $\CC$, an
  associative algebra $A$ in $\CC$, and a presentable
  $\infty$-category $\MM$ left-tensored over $\CC$, there is an
  equivalence
  \[
    \LMod_A(\CC) \otimes_\CC \MM \too{\sim} \LMod_A(\MM)
  \]
  compatible with the underlying equivalence
  $\CC \otimes_\CC \MM \to \MM$
  \cite[4.8.4.6]{lurie-higheralgebra}. On generators, this sends a
  left $A$-module $L$ in $\CC$ and an object $M$ in $\MM$ to the left
  $A$-module $L \otimes M$ in $\MM$.
\item In particular, if $\DD$ is presentable, we can take
  $\MM = \CC \otimes \DD$ and get a natural commutative diagram
  \[
    \xymatrix{
      \LMod_A(\CC \otimes \DD) \ar[d] &
      \LMod_A(\CC) \otimes \DD \ar[r] \ar[l] \ar[d] &
      \Fun^R(\DD^\op, \LMod_A(\CC)) \ar[d] \\
      \CC \otimes \DD &
      \CC \otimes \DD \ar[r] \ar[l] &
      \Fun^R(\DD^\op, \CC)
    }
  \]
  where the horizontal maps are equivalences. This allows us to
  construct a natural equivalence
  \[
    \LMod_A(\CC \otimes \DD) \to \Fun^R(\DD^\op, \LMod_A(\CC))
  \]
  lifting the equivalence $\CC \otimes \DD \to \Fun^R(\DD^\op,
  \CC)$. We refer to this as the \emph{Yoneda embedding for left
    $A$-modules}.
\item For a small $\infty$-category $S$, there is a presheaf
  $\infty$-category $\PP(S) = \Fun(S^\op, \Spa)$ which is presentable.
  For any set $\KK$ of simplicial sets such that $S$ has $\KK$-indexed
  colimits, $\PP(S)$ has a presentable subcategory
  $\PP^\KK(S) = \Fun^\KK(S^\op, S)$ of contravariant functors that
  take $\KK$-indexed colimits to $\KK$-indexed limits \cite[5.1.0.1,
  5.3.6.2, 5.5.1.1]{lurie-htt}.
\item There is a fully faithful Yoneda embedding $S \to \PP(S)$ that
  preserves all limits that exist in $S$, and the Yoneda embedding
  factors through $\PP^\KK(S)$ \cite[5.1.3.1,
  5.1.3.2]{lurie-htt}. This sends an object $X$ to a model for the
  functor $\Map_S(-,X)$.
\item These presheaf categories have the universal property that, for
  any presentable $\infty$-category $\CC$, we have
  \[
    \Fun^L(\PP^\KK(S), \CC) \cong \Fun^\KK(S, \CC)
  \]
  \cite[5.1.5.6, 5.3.6.2]{lurie-htt}.  This universal property gives
  rise to the identity
  \[
    \CC \otimes \PP^\KK(S) \cong \Fun^\KK(S^\op, \CC).
  \]
\item For a monoidal $\infty$-category $\CC$, applying the Yoneda
  embedding for left $A$-modules to $\PP^\KK(S)$  produces a natural
  equivalence
  \[
    \Ups_A\co \LMod_A(\Fun^\KK(S^\op, \CC)) \to \Fun^\KK(S^\op, \LMod_A(\CC))
  \]
  over $\Fun^\KK(S^\op, \CC)$. In short, a left $A$-module structure on
  a functor is equivalent data to a lift of the functor to left
  $A$-modules.
\end{itemize}

\section{The spectral Yoneda functor}
\label{sec:yoneda}

We can specialize the constructions and results of
\S\ref{sec:presentable} to the category of spectra. There are several
important properties of spectra that make it simplify there.

\begin{itemize}
\item For a category $S$ with pullbacks and a terminal object, there
  is a category
  \[
    \Sp(S) =
    \lim (\cdots \to S_* \overset{\Omega_\CC} \to S_*
    \overset{\Omega_\CC} \to S_* \overset{\Omega_\CC}\to \cdots)
  \]
  of ($\Omega$-)spectrum objects in $\CC$. In particular,
  \[
    \Map_{\Sp(S)}(\{X_n\}, \{Y_n\})
    \simeq \lim_n \Map_{S_*}(X_n,Y_n)
  \]
  in this category.
\item The category $\Sp$ of spectra is presentable, and for any
  presentable category $\CC$ there is an equivalence of the tensor
  product $\Sp \otimes \CC$ with the category
  \[
    \Sp(\CC) =
    \lim (\cdots \to \CC_* \overset{\Omega_\CC} \to \CC_* \overset{\Omega_\CC} \to \CC_* \overset{\Omega_\CC}\to \cdots)
  \]
  of spectrum objects in $\CC$
  \cite[4.8.1.13]{lurie-higheralgebra}. In particular, we refer to the
  equivalence
  \[
    \Ups\co \Sp(\CC) \to \Fun^R(\CC^\op, \Sp)
  \]
  as the \emph{spectral Yoneda embedding}.
\item The category $\Sp$ is idempotent under $\otimes$, and a
  presentable $\infty$-category $\EE$ is a stable presentable
  $\infty$-category if and only if $\EE \to \Sp \otimes \EE$ is an
  equivalence \cite[4.8.2.18]{lurie-higheralgebra}.
\item If $S$ is a small $\infty$-category that has pullbacks and a
  terminal object, the Yoneda embedding extends to a fully faithful
  embedding
  \[
    \Sp(S) \to \Sp(\PP(S)) \simeq \Fun(S^\op, \Sp).
  \]
  If $S$ has $\KK$-indexed colimits, this naturally factors through
  the full subcategory $\Fun^\KK(S^\op, \Sp)$.
\item Fix an associative algebra $A$ in $\Sp$. There is a full
  subcategory
  \[
    \LMod_A(\Sp(S)) \subset \LMod_A(\Fun(S^\op, \Sp)) \simeq \Fun(S^\op, \LMod_A)
  \]
  of left $A$-modules in the presheaf category whose underlying
  spectrum object lies in $\Sp(S)$. If $S$ has $\KK$-filtered
  colimits, this naturally factors through the full subcategory
  $\Fun^\KK(S^\op, \LMod_A)$.
\end{itemize}

We can assemble this into the following statement.

\begin{thm}
  \label{thm:representableaction}
  Fix an algebra $A \in \Sp$. For a small $\infty$-category $S$ which
  has pullbacks, terminal objects, and $\KK$-indexed colimits, there
  is a natural pullback diagram
  \[
    \xymatrix{
      \LMod_A(\Sp(S)) \ar[r]^-{\Ups_A} \ar[d] &
      \Fun^\KK(S^\op, \LMod_A) \ar[d] \\
      \Sp(S) \ar[r]_-\Ups&
      \Fun^\KK(S^\op, \Sp)
    }
  \]
  where the horizontal maps are fully faithful embeddings. In
  particular, a left $A$-module structure on a spectrum object $Y$
  is equivalent to a lift of the functor $F(-,Y)$ from spectra to left
  $A$-modules.
\end{thm}

\begin{prop}
  \label{prop:spectralnaturality}
  The spectral Yoneda embedding is natural in the following
  senses.
  \begin{enumerate}
  \item Suppose $S$ and $T$ are categories with pullbacks and terminal
    objects and $g\co S \to T$ is a functor with left adjoint $f$, and
    that $S$ and $T$ have $\KK$-indexed colimits.  Then there is an
    induced limit-preserving functor $g\co \Sp(S) \to \Sp(T)$ with a
    commutative diagram
    \[
      \xymatrix{ \Sp(S) \ar[r]^-\Ups \ar[d]_{g} &
        \Fun^\KK(S^\op, \Sp) \ar[d]^{f^*} \\
        \Sp(T) \ar[r]_-\Ups & \Fun^\KK(T^\op, \Sp).  }
    \]
  \item Suppose $A$ is an associative algebra in $\Sp$. Then the above
    lifts to an essentially unique commutative diagram
    \[
      \xymatrix{ \LMod_A(\Sp(S)) \ar[r]^-{\Ups_A} \ar[d] &
        \Fun^\KK(S^\op, \LMod_A) \ar[d]^{f^*} \\
        \LMod_A(\Sp(T)) \ar[r]_-{\Ups_A} & \Fun^\KK(T^\op, \LMod_A).  }
    \]
  \item Suppose $A \to B$ is a map of associative algebras in
    $\Sp$. Then the forgetful maps from $B$-modules to $A$-modules are
    part of a commutative diagram
    \[
      \xymatrix{
        \LMod_B(\Sp(S)) \ar[r]^-{\Ups_B} \ar[d] &
        \Fun^\KK(S^\op, \LMod_B) \ar[d] \\
        \LMod_A(\Sp(S)) \ar[r]_-{\Ups_A} & \Fun^\KK(S^\op, \LMod_A).
      }
    \]
  \end{enumerate}
\end{prop}

\begin{proof}
  Since $g$ is a right adjoint, it preserves pullbacks and the
  terminal object and thus induces a functor $g\co \Sp(S) \to
  \Sp(T)$. Under the spectral Yoneda embedding, for a spectrum $Y =
  \{Y_n\}$ in $\Sp(S)$ the spectrum $gY$ represents the functor $T^\op
  \to \Sp$ given by
  \[
    X \mapsto \{\Map_T(X, g Y_n)\} \simeq \{\Map_S(fX, Y_n)\}.
  \]
  This makes the functors $\Ups(gY)$ and $(\Ups Y) \circ f$ naturally
  equivalent, as desired.

  Suppose $A$ is an associative algebra in $\Sp$. For a left
  $A$-module $Y$, the object $gY \in \Sp(T)$ represents a functor
  $\Ups(gY) \simeq \Ups(Y) \circ f$ with a chosen lift
  $\Ups_A(Y) \circ f$ from $T^\op$ to left $A$-modules, and is thus
  equivalent to a left $A$-module in $\Sp(T)$.

  Suppose $A \to B$ is a map of associative algebras in $\Sp$. Then
  the final assertion is that the functor to left $B$-modules,
  represented by $Y$, forgets to the same functor to left
  $A$-modules as the underlying left $A$-module of $Y$.
\end{proof}  

\section{Stable endomorphisms}
\label{sec:stable-end}

In this section we will briefly indicate how to determine function
objects in categories of spectra. In particular, we would like to
determine information about the homotopy groups of the endomorphism
ring $\End_{\Sp(\CC)}(Y)$, as well as discuss the relationship between
these elements and stable cohomology operations.

\begin{prop}
  \label{prop:functionspectrum}
  Suppose $\CC$ is a presentable $\infty$-category and $X,Y \in
  \Sp(\CC)$. Then the function object $F_{\Sp(\CC)}(X,Y) \in \Sp$
  is a spectrum whose $k$'th term is equivalent to a homotopy limit
  \[
    \xymatrix{
      \holim_n \Map_{\CC_*}(X_{n-k}, Y_n).
    }
  \]
\end{prop}

\begin{proof}
  The function object $F(X,Y)$ has the defining property that maps
  $W \to F(X,Y)$ are equivalent to maps $W \otimes X \to
  Y$. Therefore, we can determine the structure of the spaces in the
  $\Omega$-spectrum defining $F(X,Y)$:
  \begin{align*}
    F(X,Y)_k &\simeq \Map_{\Sp}(S^{-k}, F(X,Y)) \\
        &\simeq \Map_{\Sp(\CC)}(S^{-k} \otimes X, Y)\\
        &\simeq \Map_{\Sp(\CC)}(S^{-k} \otimes (\hocolim_n S^{k-n} \otimes
          \Sigma^\infty X_{n-k}), Y)\\
        &\simeq \holim_n \Map_{\Sp(\CC)}(\Sigma^{\infty} X_{n-k}, S^{n}
          \otimes Y)\\
        &\simeq \holim_n \Map_{\CC_*}(X_{n-k}, Y_{n}).\qedhere
  \end{align*}
\end{proof}

\begin{rmk}
  Just as in the ordinary case, stable endomorphisms enjoy a
  compatibility with Mayer--Vietoris sequences. Suppose that $Y$
  is in $\Sp(\CC)$ and that we have a homotopy pushout diagram
  \[
    \xymatrix{
      A \ar[r] \ar[d] & B \ar[d] \\
      C \ar[r] & D
    }
  \]
  of objects in $\CC$. Then we obtain a homotopy pullback
  diagram
  \[
    \xymatrix{
      F(A,Y) & F(B,Y) \ar[l] \\
      F(C,Y) \ar[u] & F(D,Y) \ar[l] \ar[u]
    }
  \]
  of spectra, inducing a natural Mayer--Vietoris sequence:
  \[
    \xymatrix{
      \cdots \ar[r] &
      [\Sigma^\infty D, Y] \ar[r] &
      [\Sigma^\infty B, Y] \oplus
      [\Sigma^\infty C, Y] \ar[r] &
      [\Sigma^\infty A, Y] \ar[r] &
      \cdots\\
    }
  \]
  We know that the pullback diagram of spectra lifts to a diagram of
  left modules over $\End(Y)$; it is still a pullback diagram
  there. As a result, the natural action of the coefficient ring
  $\End(Y)_*$ on $[-,Y]$ is compatible with the connecting
  homomorphism in the Mayer--Vietoris sequence.
\end{rmk}

\begin{defn}
  For any small category $S$ and any functor $F\co S^\op \to \Sp$, let
  $\End(F)$ be the enriched endomorphism object of $F$ in $\Fun(S^\op,
  \Sp)$.
\end{defn}

\begin{prop}
  For any small category $S$ and any functor $F\co S^\op \to \Sp$, the
  functor $F$ has a natural lift to the category of left
  $\End(F)$-modules.

  If $S$ has pullbacks and a terminal object and $F$ is representable
  by a spectrum object $Y \in \Sp(S)$, we have a canonical
  identification
  \[
    \End(F) \cong \End_{\Sp(S)}(Y).
  \]
\end{prop}

\begin{proof}
  The first is true since $F$ is naturally a left
  $\End(F)$-module. The second is a consequence of the fact that the
  enriched Yoneda functor is an embedding.
\end{proof}

In particular, although $\End(F)$ generally depends on the size of the
source category when $F$ is not representable, $\End(Y)$ does not
depend on objects of $S$ other than those building $Y$. This has the
following consequence.

\begin{cor}
  Suppose that $\CC$ is a large category with pullbacks and a terminal
  object and $F\co \CC^\op \to \Sp$ is a functor represented by an
  object $Y \in \Sp(\CC)$. Then for any small subcategory
  $S \subset \CC$ containing $Y$ that is closed under pullbacks and
  the terminal object, the forgetful functor $F|_S$ has a natural lift
  to the category $\LMod_{\End(Y)}$; moreover, for $S \subset S'$
  these lifts can be made compatible.
\end{cor}

\begin{rmk}
  Unless $\CC^\op$ is presentable, to argue from this that there is a
  lift to all of $\CC^\op$ we need more, such as access to a larger
  Grothendieck universe. In this paper we are mostly concerned with
  the induced module structure on objects or small diagrams of them,
  for which this corollary is sufficient.
\end{rmk}

\begin{prop}
  \label{prop:functoriality}
  Suppose that $\CC$ and $\DD$ are categories with pullbacks and a
  terminal object, and $g\co \CC \to \DD$ is a functor with left
  adjoint $f$. Suppose that $F\co \CC^\op \to \Sp$ is represented by 
  an object $Y \in \Sp(\CC)$. Then there a map of associative algebras
  $\End(Y) \to \End(gY)$ and a commutative diagram
  \[
    \xymatrix{
      \DD^\op \ar[r]^-\Ups \ar[d]_{f^*} &
      \LMod_{\End(gY)} \ar[d] \\
      \CC^\op \ar[r]_-\Ups & \LMod_{\End(Y)}.
    }
  \]
\end{prop}

\begin{proof}
  The map $\LMod_A(\Sp(\CC)) \to \LMod_A(\Sp(\DD))$ of
  Proposition~\ref{prop:spectralnaturality} implies that the
  $\End(Y)$-module structure on $Y$ induces an $\End(Y)$-module
  structure on $gY$, or equivalently a map $\End(Y) \to \End(gY)$ of
  rings. By Theorem~\ref{thm:representableaction} we have a
  commutative diagram as follows:
  \[
    \xymatrix{
      \LMod_{\End(Y)}(\Sp(\CC)) \ar[r]^-g \ar[d]^{\Ups} &
      \LMod_{\End(Y)}(\Sp(\DD)) \ar[d]^{\Ups} &
      \LMod_{\End(gY)}(\Sp(\DD)) \ar[l] \ar[d]^\Ups \\
      \Fun(\CC^\op, \LMod_{\End(Y)}) \ar[r]_-{f^*} &
      \Fun(\DD^\op, \LMod_{\End(Y)}) &
      \Fun(\DD^\op, \LMod_{\End(gY)}) \ar[l]
    }
  \]
  The objects $Y$ and $gY$, with their actions, determine compatible
  objects in the top row. Their images in the bottom row determine
  the compatible functors desired.
\end{proof}

\section{Algebras and operads}
\label{sec:algebras-operads}

In this section, we fix an $\infty$-operad $\OO^\otimes$ and a stable
presentable $\OO$-monoidal $\infty$-category $\CC^\otimes$
\cite[3.4.4.1]{lurie-higheralgebra}. In paticular, we have a
coCartesian fibration $\CC^\otimes \to \OO^\otimes$, and for all
objects $X$ of $\OO$ the fiber $\CC_X$ is stable and
presentable. Moreover, for any objects $X_i$ and $Y$ of $\OO$ and any
active arrow $\{X_i\} \to Y$ in $\OO^\otimes$, the induced functor
$\prod \CC_{X_i} \to \CC_Y$ is colimit-preserving in each variable.

We further recall the \emph{symmetric power} functors. For any objects
$X$ and $Y$ in $\OO$ and any $d \geq 0$, there is a functor
\[
  \Sym^d_{\OO, X \to Y}\co \CC_X \to \CC_Y
\]
from \cite[3.1.3.9]{lurie-higheralgebra} (where it is denoted
$\Sym^d_{\OO,Y}$). More specifically, for each $d$ there is a map of
spaces $\PP_{X \to Y}(d) \to B\Sigma_d$ whose fiber is the space
$\Mul_\OO(X \oplus X \oplus \dots \oplus X, Y)$ of $d$-ary
active morphisms $\alpha\co (\CC_X)^d \to \CC_Y$ with its natural
$\Sigma_d$-action; the space $\PP_{X \to Y}(d)$ parametrizes $d$-fold
power functors. The symmetric power functor is the colimit:
\[
  \Sym^d_{\OO, X \to Y}(C) = \colim_{\alpha \in \PP_{X \to Y}(d)}
  \alpha(C \oplus C \oplus \dots \oplus C).
\]
The functors $\Sym^d_{\OO,X \to Y}$ preserve sifted colimits.

For any $\OO$-algebra $A$ in $\CC$, the $\OO$-algebra structure
determines action maps
\[
  \Sym^d_{\OO, X \to Y} (A(X)) \to A(Y)
\]
by \cite[3.1.3.11]{lurie-higheralgebra}. Moreover, the evaluation
functor $\ev_X\co \Alg_\OO(\CC) \to \CC_X$ has a colimit-preserving
left adjoint $\Free_{\OO,X}\co \CC_X \to \Alg_\OO(\CC)$ such that, for
any $Y$, the natural map
\[
  \coprod_{d \geq 0} \Sym^d_{\OO, X \to Y}(M) \to \Free_{\OO,X}(M)(Y)
\]
is an equivalence \cite[3.1.3.13]{lurie-higheralgebra}.

This free-forgetful adjunction gives rise to the following result.
\begin{prop}
  \label{prop:freerepresenting}
For any $M \in \CC_X$, there is a functor
\[
\Ups_M = F_{\CC_X}(M,\ev_X(-))\co \Alg_\OO(\CC) \to \Sp
\]
that is represented by a cospectrum object
\[
U(M) = \{\Free_{\OO,X}(\Omega^n M)\}\in\Sp(\Alg_\OO(\CC)^\op).
\]
\end{prop}

Let $A$ be the initial object in $\OO$-algebras. Because the
categories $\CC_X$ are pointed, there are equivalences
$A \sim \Free_{\OO,X}(*)$ for all $X$,
$A(Y) \simeq \Sym^0_{\OO,X \to Y}(M)$, and natural augmentations
\[
  \Free_{\OO,X}(M) \to A
\]
for all $M \in \CC_X$.

\begin{defn}
  For $M \in \CC_X$, define $\wt{\Free}_{\OO,X}(M)$ to be the
  fiber of the augmentation $\Free_{\OO,X}(M) \to A$.
\end{defn}

In particular, there is an equivalence
\[
  \coprod_{d > 0} \Sym^d_{\OO, X \to Y}(M) \to \wt{\Free}_{\OO,X}(M)(Y)
\]
by stability. Maps $\Free_{\OO,X}(M) \to \Free_{\OO,Y}(N)$ of objects
augmented over $A$ are the same as maps
$M \to \wt{\Free}_{\OO,Y}(N)(X)$. This gives rise to the following
description of stable natural transformations.

\begin{prop}
  \label{prop:stableformula}
For $M \in \CC_X$ and $N \in \CC_Y$, the spectrum $\StNat(\Ups_M,
\Ups_N)$ representing stable natural transformations from $\Ups_M$ to
$\Ups_N$ is of the form
\[
F_{\CC_Y}\left(N, \lim_n \Sigma^n (\wt{\Free}_{\OO,X}(\Omega^n M)(Y)\right).
\]
\end{prop}

\begin{proof}
  By Proposition~\ref{prop:functionspectrum}, in degree $k$ the
  spectrum of stable natural transformations is
  \begin{align*}
    F(
    &\Map_{\Sp(\Alg_\OO(\CC)^\op)}(U(M),U(N))\\
    &\simeq \lim_n \Map_{(\Alg_\OO(\CC)^\op)_*} (\Free_{\OO,X}(\Omega^{n-k} M),
      \Free_{\OO,Y}(\Omega^{n} N))\\
    &\simeq \lim_n \Map_{\Alg_\OO(\CC)/A}(\Free_{\OO,Y}(\Omega^n N),
      \Free_{\OO,X}(\Omega^{n-k} M))\\
    &\simeq \lim_n \Map_{\CC_Y} (\Omega^{n} N,
      \wt{\Free}_{\OO,X}(\Omega^{n-k} M)(Y))\\
    &\simeq \lim_n \Map_{\CC_Y} (\Omega^{n+k} N,
      \wt{\Free}_{\OO,X}(\Omega^{n} M)(Y))\\
    &\simeq \Map_{\CC_Y}(\Omega^k N, \lim_n \Sigma^n
      \wt{\Free}_{\OO,X}(\Omega^n M)(Y)).
  \end{align*}
  As $k$ varies these reassemble into a limit of function spectra
  \[
    F_{\CC_Y}(N, \lim_n \Sigma^n \wt{\Free}_{\OO,X}(\Omega^n M)(Y)),
  \]
  as desired.
\end{proof}

In particular, the decomposition of free algebras into symmetric
powers allows us to describe a decomposition of stable transformations
according to the number of inputs.
\begin{defn}
  \label{def:weightsummand}
  For $d > 0$, define the spectrum parametrizing natural operations of
  {\em fixed weight $d$} by
  \[
    \StNat^{\langle d\rangle}(\Ups_M, \Ups_N) \simeq F_{\CC_Y}\left(N, \lim_n \Sigma^n \Sym^d_{\OO, X \to Y}(\Omega^n M)\right).
  \]
\end{defn}

There is a natural transformation
\[
\coprod_{d > 0} \StNat^{\langle d\rangle}(\Ups_M, \Ups_N) \to 
\StNat(\Ups_M, \Ups_N) 
\]
which is an equivalence if $N$ is compact in $\CC_Y$.

\begin{prop}
  Composition multiplies weights in the following sense. For
  $X \in \CC_X$, $Y \in \CC_Y$, and $Z \in \CC_Z$ and $d,e > 0$
  there are pairings
  \[
    \StNat^{\langle e\rangle}(\Ups_N, \Ups_P) \otimes
    \StNat^{\langle d\rangle}(\Ups_M, \Ups_N) \to 
    \StNat^{\langle de\rangle}(\Ups_M, \Ups_P)
  \]
  of spectra parametrizing stable natural transformations.
\end{prop}

\begin{proof}
  Composition determines a map
  \[
    \Mul_{\OO}(\oplus^e Y,Z) \times \Mul_{\OO}(\oplus^d X,Y)^e \to 
    \Mul_{\OO}(\oplus^{de}X, Z)
  \]
  that is equivariant with respect to the action of the group of
  block-preserving permutations
  $\Sigma_d \wr \Sigma_e \subset \Sigma_{de}$. Composing with the
  diagonal on $\Mul_\OO(\oplus X^d, Y)$ makes this into a
  $\Sigma_e \times \Sigma_d$-equivariant map
  \[
    \Mul_{\OO}(\oplus^e Y,Z) \times \Mul_{\OO}(\oplus^d X,Y) \to 
    \Mul_{\OO}(\oplus^{de}X, Z),
  \]
  inducing a natural transformation of functors
  \[
    \Sym^e_{\OO,Y \to Z} \circ \Sym^d_{\OO, X \to Y} \to
    \Sym^{de}_{\OO,X \to Z}
  \]
  that preserve trivial objects. We get an induced natural transformation
  \begin{align*}
    &\lim_n(\Sigma^n \circ \Sym^e_{\OO,Y \to Z} \circ \Omega^n) \circ
    \lim_m (\Sigma^m \circ \Sym^d_{\OO,X \to Y} \circ \Omega^m)\\
    &\to \lim_{n,m}(\Sigma^n \circ \Sym^e_{\OO,Y \to Z} \circ \Omega^n) \circ
      (\Sigma^m \circ \Sym^d_{\OO,X \to Y} \circ \Omega^m)\\
    &\to \lim_n(\Sigma^n \circ \Sym^e_{\OO,Y \to Z} \circ \Omega^n) \circ
      (\Sigma^n \circ \Sym^d_{\OO,X \to Y} \circ \Omega^n)\\
    &\to \lim_n(\Sigma^n \circ \Sym^e_{\OO,Y \to Z} \circ
      \Sym^d_{\OO,X \to Y} \circ \Omega^n)\\
    &\to \lim_n(\Sigma^n \circ \Sym^{de}_{\OO,X \to Z} \circ \Omega^n).
  \end{align*}

  In addition, the functor
  $A \mapsto \lim_m \Sigma^m \Sym^e_{\OO,Y \to Z}(\Omega^m A)$
  preserves cofiber sequences, and therefore determines a stable
  functor $\CC_Y \to \CC_Z$; by stability it determines a natural map
  of function spectra
  \[
    F_{\CC_Y}(A,B) \to F_{\CC_Z}\left(\lim_m \Sigma^m \Sym^e_{\OO,Y \to
      Z}(\Omega^m A), \lim_p \Sigma^p \Sym^e_{\OO,Y \to Z}(\Omega^p
    B))\right).
  \]
  Combining these with the composition pairing for function spectra
  gives a natural map
  \begin{align*}
    &F_{\CC_Z}(P,\lim_n \Sigma^n \Sym^e_{\OO,Y \to Z} \Omega^n N)
    \otimes F_{\CC_Y}(N,\lim_m \Sigma^m \Sym^e_{\OO,Y \to Z} \Omega^m
    M)\\
    &\to F_{\CC_Z}(P,\lim_n \Sigma^n \Sym^{de}_{\OO,X \to Z} \Omega^m
    M),
  \end{align*}
  which us our desired pairing $\StNat^{\langle e\rangle}(\Ups_N,\Ups_P) \circ
  \StNat^{\langle d\rangle}(\Ups_M, \Ups_N) \to \StNat^{\langle
    de\rangle}(\Ups_M, \Ups_P)$.
\end{proof}

\begin{rmk}
  The formulas from Proposition~\ref{prop:stableformula} and
  Definition~\ref{def:weightsummand}, describing the spectra that
  parametrize stable natural transformations, bear an evident relation
  to Goodwillie's calculus of functors. We will briefly elaborate on
  this relationship.

  The free functor $\wt\Free_{\OO,X}$ and its components
  $\Sym^d_{\OO,X \to Y}$ are functors between stable
  $\infty$-categories and do not preserve colimits in general. For a
  pointed functor $F$, the assignment
  \[
    F \mapsto \lim_n \Sigma^n \circ F \circ \Omega^n
  \]
  that appears is the formula for the first \emph{coderivative} of the
  functor $F$ \cite[Lemma 2.16]{mccarthy-dualcalculus}: that is, it is
  the universal example of a functor $\mb D^1 F$ with a natural
  transformation $\mb D^1 F \to F$ such that $\mb D^1 F$ preserves
  fiber sequences.

  From this perspective, the fact that composition preserves weight is
  a reflection of the decomposition of the free functor
  $\wt\Free_{\OO,X}$ into gradings, together with the natural
  transformation
  \[
    (\mb D^1 F) \circ (\mb D^1 G) \to \mb D^1(F \circ G)
  \]
  obtained from the universal property of the first coderivative.
\end{rmk}

\section{Spaces, modules, and algebras}
\label{sec:spaces}

We now specialize the results of the previous sections to our
categories of interest.
\begin{prop}
  Fix a commutative ring spectrum $E$.
  \begin{enumerate}
  \item There is a forgetful functor $\CAlg_E \to \LMod_E$. It has a
    left adjoint $\Free_E$, given by the free commutative $E$-algebra
    on a left $E$-module.
  \item There is a functor $\Spa^\op \to \CAlg_E$ sending a space $X$ to
    the commutative $E$-algebra $E^X$. It has a left adjoint
    $\Map_{\CAlg_E}(-,E)$.
  \item The composite functor $\Spa^\op \to \CAlg_E \to \LMod_E$ has a
    composite left adjoint $\LMod_E \to \CAlg_E \to \Spa^\op$. This
    left adjoint is given by $M \mapsto Map_{\LMod_E}(M, E)$.
  \end{enumerate}
\end{prop}

\begin{proof}
  The adjunction between $E$-modules and commutative algebras is the
  specialization of the adjunction of \S\ref{sec:algebras-operads} to
  the case of the commutative $\infty$-operad. The adjunction between
  spaces and commutative algebras follows from the presentability of
  $\CAlg_E$; the limit-preserving functor
  $\Map_{\Spa}(X,\Map_{\CAlg_E(-,E)})$ is representable by a
  commutative algebra $E^X$. The composite left adjoint is the functor
  \begin{align*}
    M
    &\mapsto \Free_E(M) \\
    &\mapsto \Map_{\CAlg_E}(\Free_E(M), E)\\
    &\simeq \Map_{\LMod_E}(M, E).\qedhere
  \end{align*}
\end{proof}

\begin{cor}
  \label{cor:cospectrumrep}
  The spectrum object $E \in \Sp$ lifts to a cospectrum object
  \[
    Y = \{\Omega^n E\} \in \Sp(\LMod_E^\op),
  \]
  corepresenting the forgetful functor $\LMod_E \to \Sp$.
\end{cor}

\begin{proof}
  The composite map $\Sp(\LMod_E^\op) \to \Sp(\Spa)$ is given by
  levelwise application of the right adjoint
  $M \mapsto \Map_{\LMod_E}(M,E)$. Applied to $Y = \{\Omega^n E\}$,
  we get the spectrum $\{\Map_{\LMod_E}(\Omega^n E, E)\} \simeq E$.
\end{proof}

\begin{defn}
  Let $\Pow(E)$ be the endomorphism ring of the cospectrum
  \[
    \{\mb \Free_E(\Omega^n E)\} \in \Sp(\CAlg_E^\op)
  \]
  from Corollary~\ref{cor:cospectrumrep}. We refer to $\Pow(E)$ as the
  algebra of \emph{stable power operations} on commutative
  $E$-algebras.
\end{defn}

The underlying ring spectrum of $\Pow(E)$ is
\[
  \holim_n \Sigma^n \wt\Free_E(\Omega^n E)
\]
by Proposition~\ref{prop:stableformula}.

We recall the statement of Theorem~\ref{thm1.1}.
\begin{thm*}
  For any commutative ring spectrum $E$, the algebra of
  stable power operations fitting into a
  diagram of associative ring spectra
  \[
    E \to \Pow(E) \to \End(E).
  \]

  The ring $\Pow(E)$ has a natural action on the underlying
  spectrum of an commutative $E$-algebra in a matter compatible with
  stable cohomology operations, in the sense that there is a canonical
  lift in the following diagram:
  \[
    \xymatrix@C=4pc{
      \Sp^\op \ar[dd]_{E^{(-)}} \ar[r]^-{F(\Sigma^\infty_+(-),E)} &
      \Mod_{\End(E)} \ar[d] \\
      & \Mod_{\Pow(E)} \ar[d] \\
      \CAlg_E \ar[r] \ar@{.>}[ur] &
      \Mod_E
    }
  \]  
\end{thm*}

\begin{proof}
  The spectra $\End(E)$, $\Pow(E)$, and $E$ are, respectively,
  the endomorphism rings of the functors $\Spa^\op \to \Sp$, 
  $\CAlg_E \to \Sp$, and $\LMod_E \to \Sp$. Therefore, these functors
  automatically lift to the appropriate categories of modules, and
  thus we can apply Proposition~\ref{prop:functoriality}.
\end{proof}

\section{Fixed weight}
\label{sec:weight}

\begin{defn}
  Suppose that $\OO$ is an operad which is levelwise free and that $E$
  is a commutative ring spectrum. Then we write
  $\Op_\OO^E\langle d\rangle = \StNat^{\langle
    d\rangle}(\mathrm{id})$ for the spectrum parametrizing stable
  power operations of weight $d$ on the underlying spectra of
  $\OO$-algebras in $E$-modules (Definition~\ref{def:weightsummand}):
  \[
    \Op_\OO^E\langle d\rangle = \lim_n \Sigma^n \Sym^d_{\OO}(\Omega^n E)
  \]
  If $\OO$ is the commutative operad, we simply write $\Op^E\langle
  d\rangle$.
\end{defn}

\begin{rmk}
  An element $\theta \in \pi_k \Op_\OO^E\langle d\rangle$ induces
  compatible operations $\pi_{m} A \to \pi_{k+m} A$ as $m$ varies;
  these operations are detected by the forgetful maps
  \[
    \pi_k \Op_\OO^E\langle d\rangle \to \pi_k \Sigma^{-m}
    \Sym^d_{\OO}(\Omega^{-m} E).
  \]
\end{rmk}
Our goal in this section is to analyze $\Op_\OO^E\langle d\rangle$,
using results from Appendix~\ref{sec:equivariant-results} on
equivariant stable homotopy theory. We recall from
Definition~\ref{def:reducedrep} that $\gamma$ is the
$(d-1)$-dimensional reduced permutation representation of the
symmetric group $\Sigma_d$, and from
Proposition~\ref{prop:transitiveclassifying} that the colimit of
$S^{n\gamma}$ is a model for the classifying space $\wt{E\TT}$ for the
family of $\TT$ consisting of subgroups of $\Sigma_d$ that do not act
transitively on $\{1,\dots,d\}$.

\begin{prop}
  \label{prop:opequivariant}
  There are equivalences
  \begin{align*}
    \Op_\OO^E\langle d\rangle
    &\simeq \lim_n E \otimes [\OO(d)/\Sigma_d]^{-n\gamma}\\
    &\simeq \lim_n F_{\Sigma_d}(S^{n\gamma}, E \otimes \OO(d))\\
    &\simeq F_{\Sigma_d}(\wt{E\TT}, \OO(d)_+ \otimes E).
  \end{align*}
\end{prop}

\begin{proof}
  By definition,
  \[
    \Op_\OO^E\langle d\rangle \simeq \lim_n \Sigma^n \Sym^d_{\OO}(E
    \otimes \Omega^n S^0).
  \]
  Expanding the symmetric power functor on such a free $E$-module $E \otimes
  \Omega^n S^0$, we get
  \[
    \Sigma^n E \otimes \left[\OO(d)_+ \otimes_{\Sigma_d}
      (S^{-n})^{\otimes d}\right]\simeq E \otimes (\OO(d)_+
    \otimes_{\Sigma_d} S^{-n\gamma}),
  \]
  which is homotopy equivalent to the Thom spectrum
  \[
    E \otimes \left[\OO(d)/\Sigma_d\right]^{-n\gamma}.
  \]
  This recovers the first description.

  To recover the second, we note that the freeness of the action of
  $\Sigma_d$ on $\OO(d)$ implies that we have the Adams isomorphism
  \[
    [E \otimes S^{-n\gamma} \otimes \OO(d)_+]_{\Sigma_d} \to
    [E \otimes S^{-n\gamma} \otimes \OO(d)_+]^{\Sigma_d}.
  \]
  given by the transfer. The right-hand side is equivalent to the
  function spectrum
  \[
    F_{\Sigma_d}(S^{n\gamma}, E \otimes \OO(d)_+).
  \]
  Taking limits in $n$ gives $F_{\Sigma_d}(\wt{E\TT}, E
  \otimes \OO(d)_+)$, as desired.
\end{proof}

In particular, the fact that $\OO(d)$ is acted on freely identifies
$(E\Sigma_d)_+ \otimes \OO(d) \otimes E$ with $\OO(d) \otimes E$, and
thus we arrive at the following.
\begin{cor}
  The spectrum $\Op_\OO^E\langle d\rangle$ is naturally equivalent to
  the opposite $\TT$-Tate spectrum $(E \otimes \OO(d))^{t^\op \TT}$
  (Definition~\ref{def:tateandopposite}).
\end{cor}

\begin{thm}
\label{thm:opidentification}
  The spectrum $\Op_\OO^E\langle d\rangle$ is $p$-local if $d = p^k$ for
  some prime $p$, and is trivial otherwise.

  For any prime number $p$, there is a $p$-local equivalence between
  the spectrum of stable power operations of weight $p$ and a
  desuspended Tate spectrum:
  \[
    \Op_\OO^E\langle p\rangle \simeq \Sigma^{-1} (\OO(p) \otimes
    E)^{t\Sigma_p}_{(p)} \simeq \Sigma^{-1} \left[(\OO(p) \otimes
    E)^{tC_p}\right]^{\mb F_p^\times}
  \]
\end{thm}

\begin{proof}
  The first result is is
  Corollary~\ref{cor:transitivetatevanishing}, and the second is
  Proposition~\ref{prop:tateandopposite}.
\end{proof}

When $p=2$, this is essentially \cite[16.1]{greenlees-may-tate} and at
odd primes it is closely related to \cite[II.5.3]{bmms-hinfty}.

\begin{rmk}
  We can rephrase this in terms of Goodwillie calculus: for any $d$,
  the first coderivative $\mb D^1(\Sym^d_\OO)$ is identified with the
  opposite Tate spectrum $(E \otimes \OO(d)_+)^{t^\op \TT}$.
\end{rmk}

Our identification of the spectrum $\Op^E_\OO\langle p\rangle$ of
operations with a Tate spectrum allows us to calculate using spectral
sequence methods, because $S^{n\gamma}$ is the $n(d-1)$-skeleton in a
model for $\wt{EC}_p$. The map from stable power operations of degree $k$
and weight $d$ to power operations $\pi_m \to \pi_{k+m}$ of weight $d$
is induced by the map
\[
  \pi_k \lim_n F_{C_p}(S^{n\gamma}, E \otimes \OO(d)) \to 
  \pi_k F_{C_p}(S^{m\gamma}, E \otimes \OO(d)).
\]

\begin{rmk}
  It should be possible to assemble these into a collection of
  compatible group homology spectral sequences that eventually
  identify with the Tate spectral sequence. In the case $p=2$, this is
  more immediate: there are natural cofiber sequences
  \[
    (\Sigma_2)_+ \otimes S^{m\gamma} \to S^{m\gamma} \to S^{(m+1)\gamma}
  \]
  that give rise to a cellular filtration and a spectral sequence with
  $E_1$-term
  \[
    E^1_{p,q} = E_q(\OO(2)).
  \]
  The $E_2$-term recovers the Tate cohomology
  \[
    \widehat{H}^{1-p}(\Sigma_2; E_q(\OO(2))),
  \]
  and the spectral sequence converges to
  $\pi_{p+q} \Op_\OO^E\langle 2\rangle$. If the $E_1$-term is truncated by
  sending those terms with $p < m$ to zero, the spectral sequence
  converges to the homotopy groups of $\Sigma^m \Sym^2_\OO(\Omega^m
  E)$.
\end{rmk}

\section{Weight $p$ operations}
\label{sec:p-power}

In this section we will give example calculations with the results of
the previous sections.

\begin{exam}
  Suppose $H$ is the mod-$p$ Eilenberg--Mac Lane spectrum and $\OO$ is
  the commutative operad. The Tate spectral sequence collapses, and we
  find
  \[
    \pi_k \Op^{H\mb F_p}\langle p\rangle \cong
    \begin{cases}
      \mb F_p &\text{if }k \equiv 0,-1 \mod 2(p-1),\\
      0 &\text{otherwise.}
    \end{cases}
  \]
  In particular, when $p=2$ and $k$ is arbitrary there is a unique
  nonzero stable weight-$2$ operation $Q^k$ that increases degree by
  $k$---the Dyer--Lashof operation of the same name. This is
  represented by the nonzero element in the Tate cohomology spectral
  sequence in degree $k$, and goes to zero when we truncate the
  spectral sequence to recover operations on elements in any degree
  $m < k$.

  At odd primes, there are generating weight-$p$ operations $P^s$ in
  degree $2s(p-1)$ and $\beta P^s$ in degree $2s(p-1)-1$. This carries
  the usual warning that $\beta P^s$ is just notation: $\beta$ itself
  is not an operation, and $\beta P^s$ is not determined by $P^s$.
\end{exam}

\begin{exam}
  The case where $E = \mb S$ is the sphere spectrum is governed by the
  Segal conjecture. The spectrum of stable power operations of weight
  $p$ is $\Sigma^{-1} \mb S^{t\Sigma_p}_{(p)}$, which is the
  desuspension $\Sigma^{-1} \mb S^\wedge_p$ of the $p$-complete
  sphere. In particular, there is a certain generating stable power
  operation of degree $-1$ which we denote by $c$, and all the stable
  power operations of weight $p$ on commutative $\mb S$-algebras are
  of the form $x \mapsto \alpha c(x)$ for an element
  $\alpha \in \pi_* (\mb S)^\wedge_p$. This operation lifts $Q^{-1}$
  at the prime $2$, or $\beta P^0$ at odd primes, to an operation on
  homotopy groups.

  At the prime $2$, we can use the Tate spectral sequence to note that
  $c$ has the following properties:
  \begin{enumerate}
  \item On elements in nonnegative degrees, the operation $c$ acts
    trivially.
  \item Consider the Tate cohomology spectral sequence
    \[
    \widehat H^s(C_2; \pi_t \mb S) \Rightarrow \pi_{t-s} (\mb S)^\wedge_2.
    \]
    The generator of the $E_1$-term of the Tate spectral sequence in
    filtration $s$, which is nontrivial at $E_2$ if and only if $s$ is
    even, maps to the squaring operation
    $\pi_{-(1+s)} \to \pi_{-2(1+s)}$.
  \item In particular, the generator $c$ is represented by the
    generator of $\widehat H^0(C_2; \pi_0 \mb S)$. For a commutative
    $\mb S$-algebra $A$ and an element $x \in \pi_{-1}(A)$, we have
    $c(x) = x^2$.
  \item Further information about $c$ can be extracted from further
    information about the Tate cohomology spectral sequence: these are
    related to Mahowald's root invariants \cite{mahowald-ravenel} and
    the stable homotopy groups of stunted projective spaces. For
    example, for a commutative $\mb S$-algebra $A$ and an element
    $x \in \pi_{-2}(A)$, we have the identity $2 c(x) = \eta x^2$
    (where $\eta \in \pi_1(\mb S)$ is the Hopf invariant element); for
    a commutative $\mb S$-algebra $A$ and an element
    $x \in \pi_{-3}(A)$, we have the identities $4 c(x) = \eta^2 x^2$
    and $\eta c(x) = \nu x^2$.
  \end{enumerate}
\end{exam}

\begin{exam}
  Suppose $E$ is the complex $K$-theory spectrum $K$ and $\OO$ is the
  commutative operad. Then
  $\Sigma^{-1} E^{t\Sigma_p}_{(p)} \simeq \Sigma^{-1} K \otimes \mb
  Q_p$, the $p$-adic rationalization of $K$ by
  \cite[19.1]{greenlees-may-tate}. Even though this group of
  \emph{stable} operations is torsion-free, the action on any
  \emph{particular} homotopy group $\pi_k$ factors through a torsion
  quotient isomorphic to $\mb Q_p / \mb Z_p$, and so the operations
  always take torsion values.

  More generally, McClure has given a formula for the $p$-adic
  completion of the free $K$-algebra \cite[\S IX]{bmms-hinfty}, and
  this can be used to show that
  \[
    \Op^K\langle d\rangle \cong \begin{cases}
      K &\text{if }d=1,\\
      \Sigma^{-1} K \otimes \mb Q_p / \mb Z_p &\text{if }d = p^k,\\
      0 &\text{otherwise.}
    \end{cases}
  \]
  If we replace $K$ by its localization $K_{(p)}$, we $p$-localize the
  result, which eliminates the summands for $d \neq p^k$.
\end{exam}

\begin{exam}
  In the case of the $E_n$-operads, $\OO(p)_+$ is a finite free
  $\Sigma_p$-complex and so the Tate spectrum of $E \otimes \OO(p)_+$
  vanishes. There are no interesting stable operations for
  $E_n$-algebras.
\end{exam}

\begin{exam}
  If $E$ is a Lubin--Tate cohomology theory, we can carry out the
  above construction in the $K(n)$-local category. However, Greenlees
  and Sadofsky have shown that Tate spectra vanish in the $K(n)$-local
  category \cite{greenlees-sadofsky}, and so there are no stable
  $K(n)$-local weight-$p$ operations.
\end{exam}

\section{Operations in mod-$p$ homology}
\label{sec:homology}

In the case of ordinary mod-$p$ homology, we can make use of the
following complete calculation of the mod-$p$ homology of free
algebras.

\begin{defn}
  Let $p$ be a prime. An \emph{algebra with Dyer--Lashof operations}
  is a graded-commutative $\mb F_p$-algebra $A$ equipped with
  operations of the following type.

  \begin{description}
  \item[Case I: $p=2$.]  There are Dyer--Lashof operations
    \[
      Q^s\co \pi_kA \to \pi_{k+s} A
    \]
    satisfying the following relations.
    \begin{description}
    \item[Additivity:] $Q^s(x+y) = Q^s(x) + Q^s(y)$.
    \item[Instability:] $Q^s(x) = 0$ if $s < |x|$.
    \item[Squaring:] $Q^s(x) = x^2$ if $s = |x|$.
    \item[Unitality:] $Q^s(1) = 0$ if $s > 0$.
    \item[Cartan formula:] $Q^s(xy) = \sum_{i+j=s} Q^i(x) Q^j(y)$.
    \item[Adem relations:] $Q^r Q^s(x) = \sum \binom{i-s-1}{2i-r}
      Q^{r+s-i} Q^i x.$
    \end{description}
  \item[Case II: $p>2$.] There are operations
    \begin{align*}
      P^s\co& \pi_k A \to \pi_{k + 2s(p-1)} A,\\
      \beta P^s\co& \pi_k A \to \pi_{k+2s(p-1) - 1} A
    \end{align*}
    (written as $\beta^\epsilon P^s$ for $\epsilon \in \{0,1\}$)
    satisfying the following relations.
    \begin{description}
    \item[Additivity:] $\beta^\epsilon P^s(x+y) = \beta^\epsilon
      P^s(x) + \beta^\epsilon Q^s(y)$.
    \item[Instability:] $\beta^\epsilon P^s(x) = 0$ if $2s+\epsilon > |x|$.
    \item[Squaring:] $P^s(x) = x^p$ if $2s = |x|$.
    \item[Unitality:] $P^s(1) = 0$ if $s > 0$.
    \item[Cartan formula:]
      \begin{align*}
        P^s(xy) &= \sum_{i+j=s} P^i(x) P^j(y)\\
        \beta P^s(xy) &= \sum_{i+j=s} \beta P^i(x) P^j(y)+
                        \sum_{i+j=s} P^i(x) \beta P^j(y)\\ 
      \end{align*}
    \item[Adem relations:]
      \begin{align*}
        P^r P^s(x)
        &= \sum (-1)^{r+i} \binom{(p-1)(i-s)-1}{pi-r}
          P^{r+s-i} P^i(x)\\
        P^r \beta P^s(x)
        &= \sum (-1)^{r+i} \binom{(p-1)(i-s)}{pi-r}
          \beta P^{r+s-i} P^i(x)\\
        &- \sum (-1)^{r+i} \binom{(p-1)(i-s)-1}{pi-r-1}
          P^{r+s-i} \beta P^i(x)                           
      \end{align*}
      These identities still hold after formally applying $\beta$
      on the left and eliminating terms involving $\beta \beta$.
    \end{description}
  \end{description}
\end{defn}

\begin{thm}[{\cite[IX.2.1]{bmms-hinfty}}]
  \label{thm:freehomology}
  Let $H = H\mb F_p$ be the mod-$p$ Eilenberg--Mac Lane spectrum. Then
  for any spectrum $X$, there is a natural isomorphism between the
  homotopy of the free commutative $H$-algebra
  $\Free(H \otimes X) \simeq H \otimes \Free(X)$ and the free algebra
  $\mb Q(H_* X)$ in the category of algebras with Dyer--Lashof
  operations.

  If $\{e_i\}$ is a basis of a graded vector space $V$, then
  $\mb Q(V)$ is a free graded-commutative algebra on \emph{admissible
    monomials of excess $e(I) \geq |e_i|$}: those monomials
  $Q^{s_1} \dots Q^{s_r} e_i$ (resp.
  $\beta^{\epsilon_1} P^{s_1} \dots \beta^{\epsilon_r} P^{s_r} e_i$ if
  $p > 2$) to which the Adem relations and instability relations
  cannot be applied.
\end{thm}

\begin{thm}
  The ring $\pi_* \Pow({H})$ of stable power operations for commutative
  $H$-algebras is the completion $(\mf B^*)^\wedge$ of Mandell's
  algebra of generalized Steenrod operations with respect to the
  excess filtration.
\end{thm}

\begin{proof}
  Theorem~\ref{thm:freehomology} allows us to identify the homotopy
  groups of the sequence of spectra
  \[
    \dots \to \Sigma^n \wt\Free(\Omega^n H) \to \Sigma^{n-1}
    \wt\Free(\Omega^{n-1} H) \to \dots
  \]
  as an inverse system
  \[
    \dots \to \Sigma^n \wt{\mb Q}(\Omega^n \mb F_p) \to \Sigma^{n-1}
    \wt{\mb Q}(\Omega^{n-1} \mb F_p) \to \dots
  \]
  obtained by removing the unit. In addition, the Dyer--Lashof
  operations are stable and the products are unstable: the maps
  in this directed system annihilate all products, while preserving
  the operations $Q^s$ or $\beta^\epsilon P^s$. Therefore, this
  inverse system is equivalent to the quotient inverse system
  \[
    \dots \to \{Q^I e_n \mid e(I) \geq -n\} \to \{Q^I e_{n-1} \mid e(I) \geq -n+1\}
    \to \dots
  \]
  at $p=2$, or the analogue at odd primes. The structure maps in this
  inverse system are surjective, and so there are no
  $\lim^1$-terms. The homotopy groups of $\Pow(H)$ are then the
  completion of the group with basis
  \[
    \{Q^I e_n \mid I\text{ admissible}\}
  \]
  with respect to the excess filtration.
\end{proof}

\begin{thm}
  There is a sequence of maps of algebras
  \[
    H \to \Pow(H) \to F(H,H)
  \]
  which, on taking homotopy groups, induces the composite
  \[
    \mb F_p \to (\mf B^*)^\wedge \to \mf A^*
  \]
  from the completed algebra of generalized Steenrod operations to the
  Steenrod algebra.
\end{thm}

\begin{proof}
  Since the algebra $(\mf B^*)^\wedge$ is generated by the operations
  $Q^s$ at $p=2$, or $\beta^\epsilon P^s$ at odd primes, this follows
  from the fact that in the cohomology of spaces we have the
  identities
  \[
    Q^s(x) = \Sq^{-s}(x)
  \]
  at $p=2$ and
  \[
    \beta^\epsilon P^s(x) = \beta^\epsilon \mathrm{P}^{-s}(x)
  \]
  at $p > 2$. This is the defining property of Steenrod's reduced
  power operations \cite{steenrod-epstein}.
\end{proof}

\appendix
\section{Equivariant results}
\label{sec:equivariant-results}

In the following sections, we will need to assemble a number of
results from equivariant stable homotopy theory. Our goal is to
analyze Tate spectra constructed using families of subgroups and a
version using their opposites.

\subsection{Tate and opposite Tate spectra for families}
\begin{defn}
  \label{def:tateandopposite}
  Suppose that $G$ is a group with a family $\FF$ of subgroups. For
  any $G$-spectrum $Z$, we define the \emph{$\FF$-Tate spectrum} to be
  \[
    Z^{t\FF} =\wt{E\FF} \otimes F(EG_+, Z),
  \]
  and the \emph{opposite $\FF$-Tate spectrum} to be
  \[
    Z^{t^\op\FF} = F(\wt{E\FF}, EG_+ \otimes Z).
  \]
  If $\FF$ consists entirely of the trivial subgroup, we simply write
  $Z^{tG}$ and $Z^{t^\op G}$.
\end{defn}

\begin{rmk}
  For any subgroup $H$ of $G$, the $H$-fixed points might also be
  referred to as the $\FF$-Tate spectrum.
\end{rmk}

\begin{rmk}
  Note that the Tate and opposite Tate spectra depend only on the
  underlying spectrum $Z$ with $G$-action, and not on any
  genuine-equivariant structure.
\end{rmk}

The following result of Greenlees is referred to as Warwick duality,
and it extends \cite[16.1]{greenlees-may-tate}.

\begin{thm}[{\cite[2.5, 4.1]{greenlees-axiomatictate}}]
  \label{thm:warwickduality}
  For any group $G$ with a family $\FF$ of subgroups, and any
  $G$-spectrum $Z$, there is a natural equivalence
  \[
    \wt{E\FF} \otimes F(E\FF_+, Z) \simeq \Sigma F(\wt{E\FF}, E\FF_+
    \otimes Z).
  \]
\end{thm}

\begin{cor}
  We have a natural equivalence $Z^{tG} \simeq \Sigma Z^{t^\op G}$ for any
  group $G$ and any $G$-spectrum $Z$.
\end{cor}

\subsection{Transfer splittings}

In some cases, we can identify a Tate spectrum as a split summand of a
Tate spectrum occuring for a smaller subgroup. In this section we will
identify some conditions under which this holds.

\begin{lem}
  \label{lem:transferrestrict}
  Suppose that $Z$ is a $G$-spectrum, and that $n = [G:H]$ acts
  invertibly on $Z$. Then the composite $\Tr^G_H \Res^G_H$ acts on
  $\pi_*^G (Z^{t\FF})$ and $\pi_*^G(Z^{t^\op\FF})$ by an isomorphism.
\end{lem}

\begin{proof}
  Consider the map $\mb S \to \End(EG_+[1/n])$ of $G$-equivariant ring
  spectra. Applying $\pi_0^G$ gives the map
  $A(G) \to \pi_0((\mb S[1/n])^{hG})$, where the former is the
  Burnside ring and the latter is complete in the topology defined by
  the augmentation $I$ of the quotient $A(G) \to \mb Z$. The element
  $[G/H] \in \pi_0^G(S^0)$ maps in $A(G)/I$ to the unit
  $n \in \mb Z[1/n]$. Since $[G/H]$ reduces to a unit mod the
  augmentation ideal, and the ring is complete, $[G/H]$ is a unit in
  $\End(EG_+[1/n]).$
  
  The term $EG_+$ in the formula for Tate spectra implies that map
  $\mb S \to \End(Z^{t\FF})$ of equivariant ring spectra factors
  through the endomorphism ring $\End(EG_+[1/n])^\op$, and similarly
  $\mb S \to \End(Z^{t^\op\FF})$ factors through $\End(EG_+[1/n])$. As
  a result, in both of these rings $[G/H] \in A(G)$ maps to a
  unit. However, this element represents the endomorphism
  $\Tr^G_H \Res^G_H$.
\end{proof}

\begin{lem}
  \label{lem:weylinvariants}
  Suppose that $Z$ is a $G$-spectrum and that $H$ is a subgroup such
  that, for all $x$ not in the normalizer $N_G(H)$, the group $H \cap
  {}^x H$ is in $\FF$. Then the composite $\Res^G_H \Tr^G_H$ acts on
  $\pi_*^H(Z^{t\FF})$ and $\pi_*^H(Z^{t^\op\FF})$ by
  \[
    \alpha \mapsto N(\alpha) = \sum_{w \in W_G(H)} w \cdot \alpha.
  \]
\end{lem}

\begin{proof}
  The objects $Z^{t\FF}$ and $Z^{t^\op\FF}$ are modules over the
  smashing localization $\wt{E\FF}$, and such modules are
  characterized by the property that their $K$-fixed points are
  trivial for all $K \in \FF$. Therefore,
  $\pi_0^K(Z^{t\FF}) = \pi_0^K(Z^{t^\op\FF}) = 0$ for all $K \in \FF$.

  The double coset formula says that
  \[
    \Res^G_H \Tr^G_H\alpha = \sum_{[x] \in H\setminus G / H}
    \Tr_{H \cap {}^x H}^H (x \cdot \Res^H_{{}^{x^{-1}} H \cap H}\alpha).
  \]
  By assumption, the groups $H \cap {}^xH$ are in $\FF$ for $x$ not in
  the normalizer of $H$, and so those terms in the sum can be
  eliminated. We get a reduced formula
  \[
    \Res^G_H \Tr^G_H \alpha = \sum_{x \in W_G(H)} x \cdot \alpha
  \]
  as desired.
\end{proof}

\begin{prop}
  Suppose that $G$ is a group with a family $\mathcal{F}$ of
  subgroups, that $H$ is a subgroup such that $H \cap {}^x H$ is in
  $\FF$ for all $x$ not in the normalizer of $H$, and that $Z$ is a
  $G$-spectrum that is acted on invertibly by $[G:H]$. Then the
  restriction from $G$ to $H$ induces an isomorphism
  \[
    \pi_*^G(Z^{t\FF}) \to \left[\pi_*^H(Z^{t\FF})\right]^{W_G(H)}
  \]
  and similarly for opposite Tate spectra.
\end{prop}

\begin{proof}
  The restriction naturally maps into the $W_G(H)$-fixed
  points. Consider the triple composition
  \[
    \pi_*^G(Z^{t\FF}) \too{\Res^G_H}
    \left[\pi_*^H(Z^{t\FF})\right]^{W_G(H)} \too{\Tr_H^G}
    \pi_*^G(Z^{t\FF}) \too{\Res^G_H}
    \left[\pi_*^H(Z^{t\FF})\right]^{W_G(H)}
  \]
  By Lemma~\ref{lem:transferrestrict}, the double composite
  $\pi_*^G(Z^{t\FF}) \to \pi_*^G(Z^{t\FF})$ is an isomorphism. By
  Lemma~\ref{lem:weylinvariants}, the double composite
  $\pi_*^H(Z^{t\FF})\to \pi_*^H(Z^{t\FF})$ is the norm: however, on
  $W_G(H)$-fixed elements it is multiplication by $|W_G(H)|$, a
  divisor of $[G:H]$, and hence is an isomorphism. Therefore, both the
  restriction and transfer maps are isomorphisms. The same proof holds
  for opposite Tate spectra.
\end{proof}

\begin{cor}
  Suppose that $G$ is a group with an inclusion $\FF \subset \FF'$ of
  families of subgroups and that $H$ is a subgroup such that the
  restrictions $\FF \cap H$ and $\FF' \cap H$ of these families to $H$
  are equal. Assume that $[G:H]$ acts invertibly on a $G$-spectrum $Z$
  and that $H \cap {}^x H$ is in $\FF$ for all $x$ not in the
  normalizer of $H$. Then the induced maps
  \[
    Z^{t\FF} \to Z^{t\FF'}
  \]
  and
  \[
    Z^{t^\op\FF'} \to Z^{t^\op\FF}
  \]
  are equivalences.
\end{cor}

Combining this with the Warwick duality of
Theorem~\ref{thm:warwickduality}, we obtain the following result.
\begin{prop}
  \label{prop:tateandopposite}
  Suppose that $G$ is a group with a family $\FF$ of subgroups and
  that $H$ is a subgroup such that the restriction $\FF \cap H$
  contains only the trivial group. Assume that $[G:H]$ acts invertibly
  on a $G$-spectrum $Z$ and $H \cap {}^x H$ is trivial for all $x$ not
  in the normalizer of $H$. Then there is a shifted equivalence
  between the Tate and opposite Tate spectra:
  \[
    (Z^{t\FF})^G \simeq \left[Z^{tH}\right]^{W_G(H)} \simeq \Sigma (Z^{t^\op \FF})^G.
  \]
\end{prop}

\subsection{Localization properties}

In this section, we will show that (opposite) Tate spectra naturally
decompose as a finite direct sum over localizations.

\begin{defn}
  Suppose that $G$ is a group with a family $\FF$ of subgroups.  We
  write $d({\mathcal{F}})$ for the greatest common divisor of the
  indices $[G:H]$ for $H \in \mathcal{F}$. If $\mathcal{F}$ is
  understood, we simply write $d$.
\end{defn}

\begin{lem}
  \label{lem:indexvanishing}
  Suppose that $G$ is a group with a family $\mathcal{F}$ of
  subgroups. If $Z$ is a $G$-spectrum that is acted on invertibly by
  $d(\FF)$ then the spectra $Z^{t\FF}$ and $Z^{t^\op\FF}$ are trivial.
\end{lem}

\begin{proof}
  The map from the Burnside ring $A(G) = \pi_0^G(\mb S)$ to the
  endomorphism ring of $Z^{t\FF}$ factors through three rings:
  \begin{enumerate}
  \item $\pi_0^G \End(Z)$, where $d$ acts by a unit;
  \item $\pi_0^G \End(\wt{E\FF}) \cong \pi_0^G \wt{E\FF}$, where the
    ideal $J$ generated by the basis elements $[G/H]$ for $H \in
    \mathcal{F}$ is sent to zero; and 
  \item $\pi_0^G \End(EG_+)^\op \cong \pi_0 \mb S^{hG}$, which is complete
    in the topology generated by the augmentation ideal $I$.
  \end{enumerate}
  Therefore, the result factors through the ring
  $(A^\wedge_I)[1/d]/J$. Since the ring is a finitely generated
  abelian group, this is the completion of the ring $A[1/d]/J$ with
  respect to the topology defined by the image of $I$; however,
  $A/(I+J) = \mb Z/d$, and so this completion is trivial. Therefore,
  $\End(Z^{t\FF})$ must be the zero ring.

  The same argument applies to $Z^{t^\op\FF}$, except that we must
  use the opposite rings $\End(\wt{E\FF})^\op$ and $\End(EG_+)$.
\end{proof}

\begin{lem}
  \label{lem:tatesplitting}
  For any finite group $G$ acting on a spectrum $Z$ and any family
  $\mathcal{F}$, the natural localization maps induce decompositions
  \[
    Z^{t\FF} \simeq \bigoplus_{q | d(\FF)} (Z_{(q)})^{t\FF}
  \]
  and
  \[
    Z^{t^\op\FF} \simeq \bigoplus_{q | d(\FF)} (Z_{(q)})^{t^\op\FF}
  \]
  where the sums range over the primes $q$ dividing $d(\FF)$. In
  particular, localizing at a prime $q$ commutes with taking Tate
  spectra or opposite Tate spectra.
\end{lem}

\begin{proof}
  First note that there is a natural $G$-equivariant fiber sequence
  \[
    \bigoplus_{q | d} \Sigma^{-1} Z[1/q]/Z \to Z \to
    Z[1/d].
  \]
  We also have a fiber sequence
  \[
    \Sigma^{-1} Z[1/q]/Z \to Z_{(q)} \to Z_{\mb Q}.
  \]
  The $\FF$-Tate spectrum and opposite $\FF$-Tate spectra for $Z[1/d]$
  or $Z_{\mb Q}$ are trivial by
  Lemma~\ref{lem:indexvanishing}. Therefore,
  \[
    Z^{t\FF} \simeq \bigoplus_q \Sigma^{-1} (Z_{(q)}/Z)^{t\FF} \simeq
    \bigoplus_q (Z_{(q)})^{t\FF}
  \]
  and similarly for opposite Tate spectra.
\end{proof}

\subsection{Symmetric groups}

We can now apply these results to the symmetric groups.

\begin{prop}
  Let $\TT$ be the family of subgroups of $\Sigma_n$ that do not act
  transitively on $\{1,\dots,n\}$. If $n = p^k$ for some prime $p$,
  then $d(\TT)$ is $p$; otherwise $d(\TT) = 1$.
\end{prop}

\begin{proof}
  Any subgroup that does not act transtively is a subgroup of some
  conjugate of $\Sigma_k \times \Sigma_{n-k}$ for some $k$. The
  greatest common divisor of the indices
  $[\Sigma_n:\Sigma_k \times \Sigma_{n-k}]$ is the greatest common
  divisor of the binomial coefficients $\binom{n}{k}$ for $0 < k <
  n$. This is $p$ if $n = p^k$ and $1$ otherwise.
\end{proof}

\begin{cor}
  \label{cor:transitivetatevanishing}
  For any $\Sigma_n$-spectrum $Z$, $Z^{t\TT}$ and $Z^{t^\op \TT}$ are
  $p$-local if $n = p^k$ for some prime $p$, and trivial otherwise. If
  $Z$ is $p$-local, then these are trivial unless $n=p^k$.
\end{cor}

\begin{defn}
  \label{def:reducedrep}
  For a natural number $d$, let $\gamma$ be the $(d-1)$-dimensional
  reduced permutation representation of the symmetric group
  $\Sigma_d$, and $S^{\gamma}$ the 1-point compactification viewed as
  a $\Sigma_d$-space. We will write $S^{n \gamma}$ for the $n$-fold
  smash power, and for $P \to X$ a principal $\Sigma_d$-bundle and
  $n \in \mb Z$ we write $X^{n\gamma}$ for the Thom spectrum of the
  associated virtual bundle $n\gamma$ on $X$.
\end{defn}

In particular, as $\Sigma_d$-spaces there is an equivalence $S^1 \sma
S^{\gamma} \simeq (S^1)^{\sma d}$. The spaces $S^{n\gamma}$ form a
directed sequence of $\Sigma_d$-spaces
\[
S^0 \to S^\gamma \to S^{2 \gamma} \to S^{3 \gamma} \to \cdots,
\]
and the colimit turns out to be a classifying space for a family of
subgroups as follows.

\begin{prop}
  \label{prop:transitiveclassifying}
  Let $\TT$ be the family of subgroups $H \subset \Sigma_d$
  that do not act transitively on $\{1,\dots,d\}$. Then there is a
  cofiber sequence
  \[
    E\TT_+ \to S^0 \to \colim S^{n\gamma}
  \]
  of based spaces. In particular, the colimit is a model for
  $\wt{E\TT}$.
\end{prop}

\begin{proof}
  The inclusion of the unit sphere into the unit disc determine cofiber
  sequences
  \[
    S(n\gamma)_+ \to S^0 \to S^{n\gamma}
  \]
  that are compatible as $n$ varies. We have
  \[
    S(n\gamma)^H = S(n \gamma^H) \cong S^{n \dim(\gamma^H) - 1}.
  \]
  If $\gamma^H \neq 0$, the connectivity of this space grows in an
  unbounded fashion; if $\gamma^H = 0$, this space is empty for all
  $n$. Therefore, the colimit of the $S(n\gamma)$ is a model for
  $E\mathcal{F}$, where $\mathcal{F}$ is the family of subgroups $H$
  such that $\dim(\gamma^H) > 0$. However, by definition of $\gamma$,
  the dimension of $\gamma^H$ is one less than the number of orbits
  for the action of $H$ on $\{1,\dots,d\}$, and so $\mathcal{F} =
  \TT$.
\end{proof}

\begin{prop}
  \label{prop:tateandopposite}
  Let $W$ be the Weyl group
  $W_{\Sigma_p}(C_p) \cong \mb F_p^\times$. Then for any
  $\Sigma_p$-spectrum $Z$, the $p$-localization of the Tate spectrum
  for $\Sigma_p$ splits off from the Tate spectrum for $C_p$, and both
  are identified with a shifted opposite Tate spectrum:
  \[
    Z^{t\Sigma_p}_{(p)} \simeq (Z^{tC_p})^W \simeq \Sigma Z^{t^\op
      \Sigma_p}_{(p)}.
  \]
\end{prop}

\begin{proof}
  By Lemma~\ref{lem:tatesplitting}, it suffices to assume that $Z$ is
  $p$-local. The subgroup $C_p$ has index prime to $p$, and its
  intersection with the family $\TT$ consists only of the trivial
  group. Therefore, the result follows by
  Proposition~\ref{prop:tateandopposite}.
\end{proof}

\bibliography{masterbib}
\end{document}